\documentclass[11pt,a4paper]{article}

\usepackage{amsmath, amsfonts, amssymb}
\usepackage{theorem}
\usepackage[dvips]{epsfig}
\usepackage{latexsym}
\usepackage{exscale}
\usepackage[latin1]{inputenc}
\newtheorem{thm}{Theorem}[section]
\newtheorem{lem}[thm]{Lemma}
\newtheorem{prop}[thm]{Proposition}

\newtheorem{cor}[thm]{Corollary}

\def\be#1\ee{\begin{equation}#1\end{equation}}
\newcommand{\bea}{\begin{eqnarray}}
\newcommand{\eea}{\end{eqnarray}}
\newcommand{\beaa}{\begin{eqnarray*}}
\newcommand{\eeaa}{\end{eqnarray*}}
\newcommand{\bei}{\begin{itemize}}
\newcommand{\eei}{\end{itemize}}
\newcommand{\bee}{\begin{enumerate}}
\newcommand{\eee}{\end{enumerate}}

\def\vp{\varphi}

\def\norm#1{\left\|#1\right\|}             
\def\abs#1{\left\vert #1 \right\vert}      

\def\set#1{\left\{#1\right\}}

\def\P{{\mathbb{P}}}
\def\pr#1{\P\left(#1\right)}

\def\R{\mathbb{R}}
\def\E{\mathbb{E}}

\def\Z{{\mathbb Z}}
\def\N{{\mathbb N}}

\def\d{\, \mathrm{d}}

\def\al{\alpha}

\newcommand{\eps}{\varepsilon}

\def\H{{{\scriptstyle H}}}

\def\on{{\mathbf 1}}

\def\R{{\mathbb{R}}}

\def\s{\sigma}

\def\vp{\varphi}
\def\ro{\mathbf 0}

\def\pq{\preceq}
\def\pc{\prec}
\def\qp{\succeq}

\newenvironment{proof}[1][] {\noindent {\bf Proof#1:} }{\hspace*{\fill}$\square$\medskip\par}

\begin{document}

\title{Random Gaussian sums on trees}

\author{Mikhail Lifshits \and Werner Linde}


\date{\today}

\maketitle
\begin{abstract}
Let $T$ be a tree with induced partial order $\pq$. We investigate centered
Gaussian processes $X=(X_t)_{t\in T}$ represented as
$$
   X_t=\s(t)\sum_{v\pq t}\al(v)\xi_v
$$
for given weight functions $\al$ and $\s$ on $T$ and with $(\xi_v)_{v\in T}$ i.i.d.~standard
normal. In a first part we treat general trees and weights and derive
necessary and sufficient conditions for the a.s.~boundedness of $X$
in terms of compactness properties of $(T,d)$. Here $d$ is a special metric defined via
$\al$ and $\s$, which, in general, is not comparable with the Dudley metric generated by $X$.
In a second part we investigate the boundedness of $X$ for the binary tree
and for homogeneous weights. Assuming some mild regularity assumptions about $\al$ we completely
characterize
 weights $\al$ and $\s$ with $X$ being a.s.~bounded.
\end{abstract}
\bigskip
\bigskip
\bigskip
\bigskip
\bigskip

\vfill
\noindent
\textbf{ 2000 AMS Mathematics Subject Classification:}
Primary: 60G15; Secondary:  06A06, 05C05

\noindent
\textbf{Key words and phrases:}\
Gaussian processes, processes indexed by trees, bounded processes,
summation on trees, metric entropy
\newpage

\section{Introduction}
Let $T$ be a finite or infinite tree with root $\ro$ and let $"\pq"$ be the induced
partial order generated by the structure of $T$,
i.e., it holds $t\pq s$ or, equivalently $s\qp t$, whenever $t$ is situated on the branch connecting
$\ro$ with $s$. Suppose we are given two weight functions $\al$ and $\s$ mapping
$T$ into $[0,\infty)$ with $\s$ non--increasing, that is, $\s(t)\ge \s(s)$ whenever $t\pq s$. If
$(\xi_v)_{v\in T}$ denotes a family of independent standard normal random variables, then the
centered Gaussian process $X=(X_t)_{t\in T}$ with
\be
\label{defX}
X_t :=\s(t)\sum_{v\pq t}\al(v)\xi_v\,,\quad t\in T\,,
\ee
is well defined. Its covariance function $R_X$ is given by
\[
     R_X(t,s)=\s(t)\s(s)\sum_{v\pq t \wedge s}\al(v)^2\,,\quad t,s\in T\,.
\]
Fernique was probably the first to consider such summation schemes
on trees in his constructions of majorizing measures \cite{FerLN}.  More recently, they were extensively studied
and applied in relation to various topics, see e.g.~the literature on Derrida random energy model
\cite{BoK} or displacements in random branching walks \cite{Pem}, to mention just a few. Moreover, summation operators
related to those processes have been recently investigated in \cite{Lif10}, \cite{LifLin10} and in \cite{LL10a}.
Some of the ideas used there turned out to be useful as well for the study of Gaussian
summation schemes on trees.
\medskip

The basic question investigated in this paper is as follows: Given a tree $T$ characterize weights $\al$ and $\s$
such that $X$ is a.s.~bounded, i.e., that
\be
\label{bounded}
     \pr{\sup_{t\in T}|X_t|<\infty}=1\;.
\ee

In a first part we give necessary and sufficient conditions for the weights in order that
(\ref{bounded}) holds. These results are valid for arbitrary trees and they are based on covering
properties of $T$ by $\eps$--balls with respect to a certain metric $d$ first introduced in \cite{LifLin10}.
It is defined by
\be
\label{defd}
d(t,s):=\max_{t\pc r\pq s}\s(r)\left(\sum_{t\pc v \pq r}\al(v)^2\right)^{1/2}\,,
\ee
whenever $t\pq s$ and
we let $d(t,s):=d(t\wedge s,t)+d(t\wedge s,s)$, if $t$ and $s$ are not comparable. Here, as usual, $t\wedge s$
denotes the infimum of $t$ and $s$ in the induced partial order on $T$.
Define the covering numbers of $T$ by
$$
N(T,d,\eps) :=\inf\set{n\ge 1 : T=\bigcup_{j=1}^n B_\eps(t_j)}
$$
where $B_\eps(t_j)$ are open $\eps$--balls (w.r.t.~the metric $d$) in $T$. Then the main result of the first part is as follows:
\begin{thm}
\label{DS}
Suppose that $X$ is defined by $(\ref{defX})$ with weights $\al$ and $\s$ where $\s$ is
non--increasing. Let $d$ be the metric on $T$ given by $(\ref{defd})$. If
\be
\label{Dud}
\int_0^\infty\sqrt{\log N(T,d,\eps)}\d\eps<\infty\;,
\ee
then $X$ is a.s.~bounded.
Conversely, if  $X$ is a.s.~bounded, then necessarily
\be
\label{Sud}
\sup_{\eps>0}\eps\,\sqrt{\log N(T,d,\eps)}<\infty\;.
\ee
\end{thm}
It is worthwhile to mention that neither (\ref{Dud}) nor (\ref{Sud}) are direct consequences of the well--known
conditions due to R.M. Dudley and V.N. Sudakov (cf.~\cite{Dud} and \cite{Sud}), respectively. The latter
results are based on compactness properties of $(T,d_X)$ with so--called Dudley metric $d_X$ defined by
\be
\label{defdX}
d_X(t,s):=\left(\E\abs{X_t-X_s}^2\right)^{1/2}\,,\quad t,s\in T\;,
\ee
and not on $d$ as introduced in (\ref{defd}).
We shall see below that, in general, the covering numbers w.r.t.~$d$ and to $d_X$ may behave quite
differently. The main advantage of Theorem \ref{DS}
is that in many cases the covering numbers w.r.t.~$d$ are easier to handle than those defined by $d_X$
(cf.~\cite{LifLin10} for concrete estimates of $N(T,d,\eps)$ and also Corollary \ref{c1} below).
\medskip

It is well--known that in general entropy estimates are too rough for deciding whether or not a given
Gaussian process is bounded. Only majorizing measure techniques would work.
But, unfortunately, majorizing measures are difficult to handle and we do not see how their use leads to
a characterization of weights $\al$ and $\s$ for which $X$ is a.s.~bounded.
Therefore, in a second part, we investigate special trees and weights
where a direct  approach is possible. We suppose that $T$ is a binary tree and that the weights are
homogeneous, i.e.~$\al(t)$ and $\s(t)$ only depend
on the order $|t|$ (cf.~(\ref{deforder}) for the definition)
of  $t\in T$. Our results (cf.~Theorems \ref{t61} and \ref{t62} below) imply the following:
\begin{thm}
\label{binhom}
Let $T$ be a binary tree and suppose $\al(t)=\al_{|t|}$ and $\s(t)=\s_{|t|}$ for two
sequences $(\al_k)_{k\ge 0}$ and $(\s_k)_{k\ge 0}$ of positive numbers with $\s_k$ non--increasing.
\bee
\item
If
\be
\label{regular}
\sup_n \sup_{n\le k\le 2n}\frac{\al_k}{\al_n}<\infty\;,
\ee
then $X$ defined by $(\ref{defX})$ is a.s.~bounded if and only if
\be
\label{condsum}
\sup_n \s_n\sum_{k=1}^n \al_k<\infty\;.
\ee
In particular, if the $\al_k$ are non--increasing, then $(\ref{regular}) $ is always satisfied, hence in
that case $X$ is a.s.~bounded if and only if $(\ref{condsum})$ is valid.
\item
If the $\al_k$ are non--decreasing, then $X$ is a.s.~bounded if and only if
$$
\sup_n \s_n \sqrt n \left(\sum_{k=0}^n \al_k^2\right)^{1/2}<\infty\;.
$$
\eee
\end{thm}
\medskip

The organization of this paper is a s follows. After a short introduction to trees,  Section \ref{ProofDS} is devoted
to the proof of  Theorem \ref{DS}.
In Section \ref{metrics} we thoroughly investigate the relation between the two metrics $d$ and $d_X$. 
Here the main observation is that $N(T,d,\eps)$ and $N(T,d_X,\eps)$ may behave
quite differently. Nevertheless, in view of Theorem \ref{DS} and the well--known results
due to R.M. Dudley and to V.N. Sudakov, on the logarithmic level the covering numbers of
these two metrics should be of similar order. We investigate this question in Section
\ref{s:more_ent} more thoroughly. In particular, we show that
$\eps^2\log N(T,d,\eps)$ is bounded if and only if $\eps^2\log N(T,d_{X},\eps)$ is so.
In Section \ref{Binary} we treat processes $X$ indexed
by a binary tree and with homogeneous weights. We prove slightly more general results than
stated in Theorem \ref{binhom}. Finally, we give some interesting examples of bounded as well as
unbounded processes indexed by a binary tree.
In particular, these examples show that the boundedness of $X$ may not be described by properties
of the product $\al\,\s$ only.

\section{Trees}
\setcounter{equation}{0}
Let us recall some basic notations related to trees which will be used later on.
In the sequel $T$ always denotes a finite or an infinite
tree. We suppose that $T$ has a unique root which we denote by
$\ro$ and that each element $t\in T$ has a finite number
of offsprings. Thereby we do not exclude that some elements do not possess any offspring, i.e.,
the progeny of some elements may "die out". The tree structure leads in natural way
to a partial order $,\!,\pq "$ by letting $t\pq s$, respectively $s\qp t$,  provided there are
$t=t_0, t_1,\ldots, t_m=s$ in $T$ such that for $1\le j\le m$ the
element $t_j$ is an offspring of $t_{j-1}$.
The strict inequalities have the same meaning with the additional assumption $t\not=s$.
Two elements $t,s\in T$ are said to be comparable provided that either $t\pq s$ or $s\pq t$.

For $t,s\in T$ with $t\pq s$ the order interval $[t,s]$ is defined by
$$
[t,s]:= \set{ v\in T : t\pq v\pq s}
$$
and in a similar way we construct $(t,s]$ or $(t,s)$ .

A subset $B\subseteq T$ is said to be a branch provided that all elements in $B$ are
comparable and, moreover, if $t\pq v\pq s$ with $t,s\in B$, then this implies $v\in B$
as well. Of course, finite branches are of the form $[t,s]$ for suitable $t\pq s$.

For any $s\in T$ its order $|s|\ge 0$ is defined by
\be
\label{deforder}
|s|:=\#\set{t\in T :  t \pc s}\;.
\ee
Let $\rho$ be an arbitrary metric on the tree $T$.
Given $\eps>0$ a set $\mathcal O\subseteq T$ is said to be an $\eps$--order net w.r.t.~$\rho$
provided that for each $s\in T$ there is an $t\in \mathcal O$ with $t\pq s$ and $\rho(t,s)<\eps$. Let
\be
\label{ordercov}
\tilde N(T,\rho,\eps):=\inf\set{\# \{\mathcal O\} : \mathcal O\;\mbox{is
an}\:\eps\mbox{--order net of}\;T}
\ee
be the corresponding order covering numbers. Clearly, we have
$$
N(T,\rho,\eps)\le \tilde N(T,\rho,\eps)\;.
$$
As shown in \cite[Proposition 3.3]{LifLin10},
for the metric $d$ defined by (\ref{defd}) we
also have a reverse estimate.
More precisely, here it always holds
\be
\label{tildeN}
\tilde N(T,d,2\eps)\le N(T,d,\eps)\;.
\ee
\section{Proof of Theorem \ref{DS}}
\setcounter{equation}{0}
\label{ProofDS}
Let $T$ be an arbitrary tree and let $\al$ and $\s$ be weights on $T$ as before. Define $X=(X_t)_{t\in T}$
as in (\ref{defX}).
Of course, whenever (\ref{bounded}) holds, then we necessarily have
\be
\label{cond1}
\sup_{t\in T}\left(\E |X_t|^2\right)^{1/2}=\sup_{t\in T}\s(t)\left(\sum_{v\pq t}\al(v)^2\right)^{1/2}<\infty\;.
\ee
Thus let us assume that (\ref{cond1}) is always satisfied.
\bigskip

In order to prove part one of Theorem \ref{DS} in a first step we replace the
process $X$ by a process $\hat X$ which is easier to handle.

To this end, if $k\in\Z$, define $I_k\subseteq T$ by
\be
\label{defIk}
    I_k:=\set{t\in T : 2^{-k-1}<\s(t)\le 2^{-k}}
\ee
and a new weight $\hat\s$ by
\be \label{defhats}
    \hat\s:=\sum_{k\in \Z} 2^{-k}\on_{I_k}\;.
\ee
Let $\hat X$ be the process defined by $\al$ and $\hat\s$ via (\ref{defX}), i.e., it holds
\be
\label{defhatX}
\hat X_t=\hat\s(t)\sum_{v\pq t}\al(v)\xi_v\;,\quad t\in T\;,
\ee
and let $\hat d$ denote
the distance generated via $\al$ and $\hat \s$ as in (\ref{defd}). Then the following are valid.

\begin{prop}~
\label{p1}
\bee
\item
If $t\pq s$, then it holds
$$
d(t,s)\le \hat d(t,s)\le 2\,d(t,s)\;.
$$
Consequently, it follows
$$
\tilde N(T,d,\eps)\le \tilde N(T,\hat d,\eps)\le \tilde N(T,d,\eps/2)\;,
$$
where $\tilde N(T,d,\eps)$ and $\tilde N(T,\hat d,\eps)$
are the order covering numbers corresponding to the respective metrics.
\item
The process $X$ is a.s.~bounded if and only if $\hat X$ is a.s.~bounded.
\eee
\end{prop}

\begin{proof}
The first assertion follows easily by the definition of $d$ and $\hat d$ while the second one
is a direct consequence of
$$
|X_t|\le |\hat X_t|\le 2\,|X_t|\,,\quad t\in T\;.
$$
\end{proof}
\medskip

As a consequence of the preceding proposition we conclude that it suffices to prove
Theorem \ref{DS} in the case of non--increasing weights $\s$ of the form
\be
\label{sig}
\s:=\sum_{k\in \Z}2^{-k}\on_{I_k}
\ee
The property that $\s$ is non--decreasing reflects in the following properties of the partition
$(I_k)_{k\in\Z}$ of $T$.
\begin{enumerate}
\item
Whenever $B\subseteq T$ is a branch, then for each $k\in\Z$ either $B\cap I_k=\emptyset$
or it is an order interval in $T$.
\item
If $l< k$, $t\in B\cap I_l$, $s\in B\cap I_k$, then this implies $t\pc s$.
\item
$I_k=\emptyset$ whenever $k\le k_0$ for a certain $k_0\in\Z$.
\end{enumerate}
Thus from now on we may suppose that the weight $\sigma$  is as in (\ref{sig}) with a partition $(I_k)_{k\in\Z}$ of $T$
possessing properties (1), (2) and (3) stated before.
\bigskip

In a second step of the proof of Theorem \ref{DS}, first part, we define a process $Y:=(Y_t)_{t\in T}$
which may be viewed as a localization of $X$. To this end let us write
$t\equiv s$ provided there is a $k\in\Z$
such that  $t,s\in I_k$.
With this notation  we set
\be
\label{defY}
     Y_t:=\s(t)\sum_{v\pq t\atop{v\equiv t}}\al(v)\xi_v\;,\quad t\in T\;.
\ee
It is an easy deal to relate the boundedness of $X$ with that of $Y$.
\begin{prop}
\label{p2}
     The process $Y$ is a.s.~bounded if and only if $X$ is a.s.~bounded.
\end{prop}

\begin{proof}
Actually, we establish simple linear relations between $Y$ and $X$, see (\ref{YtoX})
and (\ref{XtoY}) below.
For any integers $\ell\le k$ and any $t\in I_k$ set $B_\ell(t):=[\ro,t]\cap I_\ell$.
Then we have
\begin{eqnarray}
 \nonumber
  X_t &=& 2^{-k} \sum_{\ell\le k} \sum_{v\in B_\ell(t)} \al(v)\xi_v
  \\  \nonumber
  &=&  \sum_{\ell\le k}  2^{-(k-\ell)} \cdot 2^{-\ell} \sum_{v\in B_\ell(t)} \al(v)\xi_v
  \\ \label{YtoX}
  &=& \sum_{\ell} \ 2^{-(k-\ell)} Y_{\lambda_\ell(t)},
\end{eqnarray}
where the last sum is taken over $\ell\le k$ such that $B_\ell(t)\not=\emptyset$
and $\lambda_\ell(t):=\max\{s:s\in B_\ell(t)\}$.
It follows from (\ref{YtoX}) that the boundedness of $Y$ yields that of $X$.

To prove the converse statement of Proposition \ref{p2},
take an arbitrary $t\in T$ and consider two different cases.\\
If $t\equiv \ro$ (recall that $\ro$ denotes the
root of $T$), then by the definition of $Y$ we simply have
$Y_t=X_t$.

Otherwise, if $t\not\equiv\ro$, let
\[
  \lambda^-(t)=\max\{s: s\pq t, s\not\equiv t\}.
\]
By the definition of $Y_t$ we obtain
\begin{eqnarray} \nonumber
Y_t &=& \s(t)\sum_{\lambda^-(t) \prec v \pq t} \al(v)\xi_v
\\ \nonumber
&=& \s(t) \left( \sum_{ v \pq t} \al(v)\xi_v
               - \sum_{v\pq \lambda^-(t)} \al(v)\xi_v\right)
\\ \label{XtoY}
 &=& X_t -\frac{\s(t)}{\s(\lambda^-(t))}\ {X_{\lambda^-(t)}}.
\end{eqnarray}
Since the weight $\s$ is non--increasing, by $\lambda^-(t)\pq t$ we get
$\frac{\s(t)}{\s(\lambda^-(t))}\le 1$.
It follows from (\ref{XtoY}) that if $X$ is a.s.~bounded this is also valid for $Y$ as claimed.
This completes
the proof.
\end{proof}
\medskip

In the next step we calculate the Dudley distance generated by $Y$ and compare $N(T,d_Y,\eps)$
with $\tilde N(T,d,\eps)$. Recall that $ \tilde N(T,d_Y,\eps)$ and $\tilde N(T,d,\eps)$ are the
corresponding order covering numbers as introduced in $(\ref{ordercov})$.

\begin{prop}
\label{p4}
Suppose $\al(\ro)=0$, hence $Y_\ro=0$ a.s. Then it follows that
\be
\label{entrest}
N(T,d_Y,\eps)\le \tilde N(T,d_Y,\eps)\le\tilde N(T,d,\eps)+1\;.
\ee
\end{prop}

\begin{proof}
If $t\pq s$, then we get
$$
d_Y(t,s)^2= \s(s)^2\sum_{t\pc v\pq s}\al(v)^2 = d(t,s)^2    \qquad\mbox{if}\quad t\equiv s
$$
and
$$
   d_Y(t,s)^2=\E|Y_t|^2+\E|Y_s|^2\quad\mbox{if}\quad t\not\equiv s\;.
$$
Given $\eps>0$ let $\mathcal O\subseteq T$
be an $\eps$--order net w.r.t.~the metric $d$. Take $s\in T$ arbitrarily. Then there is
a $t\in \mathcal O$ such that $t\pq s$ and $d(t,s)<\eps$. If $t\equiv s$, then this implies
$d_Y(t,s)=d(t,s)<\eps$ as well. But if $t\not\equiv s$,
then we get
\beaa
 d_Y(\ro,s)&=& \left(\E|Y_s-Y_\ro|^2\right)^{1/2}
            = \left(\E|Y_s|^2\right)^{1/2}
            =\s(s) \left(\sum_{v\pq s\atop{v\equiv t}}\al(v)^2\right)^{1/2}
 \\
            &\le& \s(s)\left(\sum_{t\pc v\pq s}\al(v)^2\right)^{1/2}
            \le d(t,s)
            <\eps\,.
\eeaa
In different words, the set $\mathcal O\cup\{\ro\}$ is an $\eps$--order net of $T$ w.r.t.~$d_Y$.
Of course, this implies the second inequality in (\ref{entrest}),
the first one being trivial.
Thus the proof is complete.
\end{proof}

\noindent
\textbf{Proof of Theorem \ref{DS}, first part}:
Without loosing generality we may assume $\al(\ro)=0$. Indeed,
write
$$
    X_t=\s(t)\sum_{v\pq t}\al(v)\xi_v=\s(t)\sum_{\ro\pc v\pq t}\al(v)\xi_v+ \s(t)\al(\ro)\xi_\ro
$$
and observe that $\sup_{t\in T}\s(t)<\infty$. Moreover, the metric $d$ is independent of $\al(\ro)$.
Note that this number never appears in the evaluation of $d(t,s)$ for arbitrary $t,s\in T$.

Thus let us assume now that (\ref{Dud}) is valid. Then (\ref{tildeN})
implies $$\int_0^\infty \sqrt{\log \tilde N(T,d,\eps)}\d\eps<\infty$$ as well. Hence
Proposition \ref{p4} yields
$$
\int_0^\infty \sqrt{\log N(T,d_Y,\eps)}\d\eps<\infty\;.
$$
Consequently,
Dudley's theorem (cf.~\cite{Dud} or \cite{Lif}, p.179) applies for $Y$ and $d_Y$, hence
$Y$ possesses a.s.~bounded paths. In view of Proposition \ref{p2}
the paths of $X$ are also a.s.~bounded and this completes the proof of the first part
of Theorem \ref{DS}.
\medskip\par

\noindent
\textbf{Proof of Theorem \ref{DS}, second part}:
Take $\eps>0$. As proved in \cite[Proposition 5.2]{LifLin10}  there are at least
$m:=N(T,d,2\eps)-1$ disjoint order intervals $(t_i,s_i]$ in $T$ with
$d(t_i,s_i)\ge \eps$. By the definition of $d$ we find $t_i\pc r_i\pq s_i$ such that
$$
\s(r_i)\left(\sum_{t_i\pc v\pq r_i}\al(v)^2\right)^{1/2}\ge \eps\,,\quad 1\le i\le m\;.
$$
Next, set
\be\label{defeta}
\eta_i:= X_{r_i}-\frac{\s(r_i)}{\s(t_i)}\,X_{t_i}\,,\quad 1\le i\le m\;.
\ee
Then it follows that
$$
\eta_i=\s(r_i)\left[\sum_{v\pq r_i}\al(v)\xi_v -\sum_{v\pq t_i}\al(v)\xi_v\right]
=\s(r_i)\left[\sum_{t_i\pc v\pq r_i}\al(v)\xi_v\right]
$$
and, consequently, the $\eta_i$ are independent centered Gaussian with
\be \label{esteta0}
 (\E\abs{\eta_i}^2)^{1/2}=\s(r_i)\left(\sum_{t_i\pc v\pq r_i}\al(v)^2\right)^{1/2}\ge \eps\;.
\ee
Since $\s$ is assumed to be non--increasing, we get
\be
\label{esteta}
  \sup_{1\le i\le m}\abs{\eta_i}\le 2\sup_{t\in T}\abs{X_t}\;.
\ee
Suppose now that $X$ is a.s.~bounded. By Fernique's theorem (cf.~\cite{Fer} or~\cite{Lif}, p.142)
this implies
$$
  C:=\E \sup_{t\in T}|X_t|<\infty\;,
$$ hence (\ref{esteta}) leads to
$$
\E\sup_{1\le i\le m}\abs{\eta_i}\le 2\, C\;,
$$
and by the choice of $m$ the assertion follows by
$$
c\,\eps\,\sqrt{\log m}\le\E\sup_{1\le i\le m}\abs{\eta_i}
$$
where we used (\ref{esteta0}) and the classical Fernique--Sudakov bound recalled below
in (\ref{sudakov}).
\hspace*{\fill}$\square$\medskip\par
\medskip

The main advantage of Theorem \ref{DS} is that there are quite general techniques to get
precise estimates for $N(T,d,\eps)$ (cf.~\cite{LifLin10}). For example, Theorem \ref{DS} implies the following.
\begin{cor}
\label{c1}
Let $T$ be a binary tree and suppose that
$$
\al(t)\s(t)\le c\,|t|^{-\gamma}\,,\quad t\in T\,,
$$
for some $\gamma>1$.  Then
$X$ defined by $(\ref{defX})$ is a.s.~bounded.
Conversely, if
$$
\al(t)\ge c\,|t|^{-\gamma}
$$
for some $\gamma<1$ and $\s(t)\equiv 1$, then the generated process $X$ is a.s.~unbounded.
\end{cor}
\begin{proof}
As shown in \cite{LifLin10}, an estimate
$
\al(t)\s(t)\le c\,|t|^{-\gamma}$ implies
$\log N(T,d,\eps)\le c\, \eps^{-2/(2\gamma-1)}$
for each $\gamma>1/2$.
 Hence, if $\gamma>1$, then (\ref{Dud}) holds, hence Theorem \ref{DS} applies and completes the
proof of the first part.

The second part follows by $\log N(T,d,\eps)\ge c\, \eps^{-2/(2\gamma-1)}$
whenever $\al(t)\ge c\,|t|^{-\gamma}$ for some $\gamma>1/2$ and $\s(t)\equiv 1$
(cf.~\cite[Proposition 7.7]{LifLin10}). Thus, by Theorem \ref{DS} the
process $X$ cannot be bounded if $\gamma<1$.
\end{proof}

\begin{remark}
\rm
The second part of Corollary \ref{c1} does no longer hold for non--constant weights $\s$. In different words,
an estimate $\al(t)\s(t)\ge c\,|t|^{-\gamma}$  with $1/2<\gamma<1$ does not always imply that $X$ is unbounded (cf.~the
remark after Corollary \ref{c2} below).
\end{remark}
\medskip

Corollary 3.4 suggests that the boundedness of a process with weights
$\al(t)$ and $\s(t)$ might be determined  by the product $\sigma(t)\alpha(t)$.
In other words, it is natural to ask what is the relation between the boundedness
of this process and the process generated by the weights $\widetilde{\sigma}(t):\equiv 1$
and $\widetilde{\alpha}(t):=\sigma(t)\alpha(t)$. It turns out that (only) a one--sided
implication is valid.
\medskip

To investigate this question, for a moment write $X^{\al,\s}$ for the process defined in (\ref{defX}).
\begin{prop} \label{al_s_1_als}
If $X^{\al\s,1}$ is a.s.~bounded, then this is also true
for $X^{\al,\s}$.
\end{prop}
\begin{proof}
We only give a sketch of the proof.
\begin{enumerate}
\item Recall a general fact from the theory of Gaussian processes: If $X$ and $Y$ are two independent centered
Gaussian processes, then $X+Y$ bounded yields $X$ bounded. This is an immediate consequence
of Anderson's inequality (cf. \cite[p.135]{Lif}).

\item  By applying this fact we obtain: If two weights are related by $\al_1\le c\,\al_2$ and
$X^{\al_2,1}$ is bounded, then $X^{\al_1,1}$ is bounded as well.

\item Suppose now that $X^{\al\s,1}$ is bounded, then  $X^{\al\hat\s,1}$  is bounded,
where the binary weight
$\hat \s$ is defined  in (\ref{defhats}). Set $X':= X^{\al\hat\s,1}$.

\item Let now $Y$ be the process constructed in the article  associated to  $X^{\al,\s}$.
Since $Y_t=X_t'-X_{\lambda^-(t)}'$, we see that if $X'$ is bounded, then $Y$ is bounded.

\item Recall that we know that the boundedness of $Y$ is equivalent to that of $X^{\al,\s}$.
\end{enumerate}
\end{proof}
Examples in Section \ref{Binary} show that the statement of Proposition \ref{al_s_1_als}
cannot be reversed, i.e., in general the boundedness of
$X^{\al,\s}$ does not yield that of $X^{\al\s,1}$.

\section{Compactness properties of $(T,d)$ versus those of $(T,d_X)$}
\label{metrics}
\setcounter{equation}{0}

The aim of this section is to compare the metric $d$ on $T$ defined in (\ref{defd}) and
the Dudley distance $d_X$ introduced in (\ref{defdX}). For $X$ defined by (\ref{defX}) the latter distance
equals
$$
d_X(t,s)^2= \abs{\s(t)-\s(s)}^2\sum_{v\pq t\wedge s}\al(v)^2 + \s(t)^2\sum_{t\wedge s \pc v\pq t}\al(v)^2
+ \s(s)^2 \sum_{t\wedge s\pc v\pq s}\al(s)^2\;.
$$
In particular, if $t\pq s$, this reads as
\be
\label{dX2}
d_X(t,s)^2= \abs{\s(t)-\s(s)}^2\sum_{v\pq t}\al(v)^2 + \s(s)^2\sum_{t\pc v\pq s}\al(v)^2\;.
\ee
Recall, that for $t\pq s$ we have
\be
\label{defd2}
d(t,s)=\max_{t\pc r\pq s}\s(r)\left(\sum_{t\pc v \pq r}\al(v)^2\right)^{1/2}\,.
\ee
Comparing (\ref{dX2}) with (\ref{defd2}), it is not clear at all how these two distances are related in general.
\medskip

In a first result we show that the covering numbers w.r.t.~$d$ and to $d_X$ may be of quite different order.
\begin{prop}
\label{ex1}
There are non--increasing weights $\al$ and $\s$ on a tree $T$ such that the
generated process $X$ is a.s.~bounded and, moreover,
$$
\lim_{\eps\to 0}\frac{N(T,d_X,\eps)}{N(T,d,\eps)}=\infty\;.
$$
\end{prop}
\begin{proof}
Take $T=\N_0=\set{0,1,\ldots}$ and let $\al(0)=\s(0)=1$. If $k\ge 1$
set
$$
\al(k)=k^{-\nu}\quad\mbox{and}\quad \s(k)=k^{-\theta}
$$
for some $\theta,\nu>0$, i.e.,
\be
\label{defXk}
X_k=k^{-\theta}\left[\sum_{j=1}^k j^{-\nu}\,\xi_j +\xi_0\right]\;,\quad k\ge 1\;.
\ee
The law of iterated logarithm tells us that the process $X$ is a.s.~bounded if and only if
$\theta+\nu> 1/2$. Thus let us assume that this is satisfied.

Take now any $1\le k<l$. Then
by (\ref{dX2}) it follows
$$
d_X(k,l)\ge k^{-\theta}-l^{-\theta}\ge
k^{-\theta} - (k+1)^{-\theta}
          \ge c_\theta\, k^{-\theta-1}\;.
$$
Hence, if $1\le k<l\le n$ for some $n\ge 2$, this implies
$$
d_X(k,l)\ge c_\theta n^{-\theta-1}
$$
which yields
\be
\label{dXest}
N(T,d_X,\eps)\ge c\, \eps^{-1/(\theta+1)}
\ee
for some $c>0$ only depending on $\theta$.

On the other hand, we have $\al(k)\s(k)= k^{-(\theta+\nu)}$. As shown in \cite[ Proposition 6.3]{LifLin10}
(apply this proposition with $q=2$, $\H=0$ and $\gamma=2(\theta +\nu)$) a bound
$\al(k)\s(k)\le k^{-(\theta+\nu)}$ implies
\be
\label{upperentr}
N(T,d,\eps)\le c\,\eps^{-1/(\theta+\nu)}\;.
\ee
Of course, if $\nu>1$, then (\ref{dXest}) and (\ref{upperentr}) lead to
\be
\label{different}
\lim_{\eps\to 0}\frac{N(T,d_X,\eps)}{N(T,d,\eps)}=\infty
\ee
completing the proof.
\end{proof}
\medskip

Let us state some interesting consequence of the preceding proposition. To this end recall a result due
to M. Talagrand (cf. \cite{Ta} and \cite{Led}).
Suppose $X=(X_t)_{t\in T}$ is a centered Gaussian process
on an arbitrary index set $T$ and let $d_X$, as in (\ref{defdX}), be the Dudley metric on $T$ generated by $X$.
If $N(T,d_X,\eps)\le \psi(\eps)$ for a non--increasing function $\psi$
satisfying
\be
\label{condpsi}
c_1\,\psi(\eps)\le \psi(\eps/2)\le c_2\,\psi(\eps)
\ee
for certain $1<c_1<c_2$, then this implies
\be
\label{Talest}
-\log\pr{\sup_{t\in T}\abs{X_t}<\eps}\le c\,\psi(\eps)
\ee
for some $c>0$.

We claim now that in the case of processes $X$ defined by (\ref{defXk}) even holds
\be
\label{sbsum}
-\log\pr{\sup_{k\ge 1}\abs{k^{-\theta}\left[\sum_{j=1}^k\,j^{-\nu}\xi_j+\xi_0\right]}<\eps}\approx \eps^{-1/(\theta+\nu)}\;.
\ee
Indeed,
if we apply Proposition 7.1 in \cite{LifLin10} with $\vp(x)=x^{-\gamma}$ where $\gamma=2(\theta+\nu)$, we see
that estimate (\ref{upperentr}) is sharp, i.e., we obtain
$$
N(T,d,\eps)\approx
\eps^{-1/(\theta+\nu)}\;.
$$
Consequently, (\ref{sbsum}) follows by Proposition 9.1 in \cite{LifLin10}.

Comparing (\ref{sbsum}) with (\ref{dXest})
shows that for $\nu>1$ estimate (\ref{Talest}) cannot lead to sharp estimates while, as seen above, the
use of $N(T,d,\eps)$ does so. In some sense this observation proves that the metric $d$ fits better to those processes $X$ than $d_X$
does.
\bigskip

One may ask now whether or not there are examples of trees and weights such that the quotient in
(\ref{different}) tends to zero, i.e., whether there are examples with
\be
\label{different1}
\lim_{\eps\to 0}\frac{N(T,d,\eps)}{N(T,d_X,\eps)}=\infty\;.
\ee
Although we do not know the answer to this question let us shortly indicate why such examples
are hardly to construct provided they exist. Indeed, if $N(T,d,\eps)\approx \eps^{-a}\abs{\log\eps}^b$
for some $a>0$ and $b\ge 0$, then by Proposition 9.1 in \cite{LifLin10} this implies
$$
-\log\pr{\sup_{t\in T}\abs{X_t}<\eps}\approx \eps^{-a}\abs{\log\eps}^b\;.
$$
Consequently, whenever $N(T,d_X,\eps)\approx \psi(\eps)$ with $\psi$ satisfying (\ref{condpsi}), then
by (\ref{Talest}) we get
$$
N(T,d,\eps)\le c\,\psi(\eps)\le c'\,N(T,d_X,\eps)\;,
$$
hence in that situation examples satisfying (\ref{different1}) cannot exist.
\medskip

In spite of this observation we will show now that $d_X(t,s)$ may become arbitrarily small
while $d(t,s)\ge C>0$. Hence an estimate $d(t,s)\le c\,d_X(t,s)$ cannot be valid in general. Recall
that in view of Proposition \ref{ex1} a relation $d_X(t,s)\le c\,d(t,s)$ is impossible as well.
\begin{prop}
\label{ex2}
There are weights $\al$ and $\s$ on $T=\N_0$ such that the corresponding process $X$ is a.s.~bounded
and such that $\lim_{k\to\infty} d_X(\ro,k)=0$
while $d(\ro,k)=C>0$ for all $k\ge 1$.
\end{prop}
\begin{proof}
For $k\in\N_0$ choose $\s(k)=2^{-k}$ while $\al(0)=0$ and $\al(k)=k^{-1}$
for $k\ge 1$. Of course, the generated process $X$ is a.s.~bounded. Moreover, if $k\ge 1$, then it follows that
$$
d_X(\ro,k)= 2^{-k}\left(\sum_{v=1}^k v^{-2}\right)^{1/2}\;.
$$
In particular, $d_X(\ro,k)\to 0$ quite rapidly as $k\to\infty$. On the other hand,
$$
d(\ro,k)= 2^{-1}\al(1)=2^{-1}
$$
and this completes the proof with $C=2^{-1}$.
\end{proof}

\section{More about the relation between processes and their entropy}
\setcounter{equation}{0}
\label{s:more_ent}

So far, we came out with a somewhat messy set of relations  between the processes
and their entropy. Let us try to rearrange it again and to put the things in order.
We have three consecutively generated processes $X\to\hat X\to Y$. In this section, we will not identify
or replace $X$ by $\hat X$, as we did sometimes before.

Before proceeding further, let us make a useful and well--known identification. Suppose $X=(X_t)_{t\in T}$ is an arbitrary
Gaussian process (in fact we only need that it is a process with finite second moments) modeled over a probability
space $(\Omega,\mathcal A,\P)$. Then we may regard $X$ as subset of the Hilbert space $L_2(\Omega,\mathcal A,\P)$, i.e.,
 we identify $X$ with $\set{ X_t : t\in T}$ and the induced distance equals
$$
\norm{X_t-X_s}_2=\left(\E|X_t-X_s|^2\right)^{1/2}=d_X(t,s)\;.
$$
In particular, we may also build the absolutely convex hull of $X$ in $L_2(\Omega,\mathcal A,\P)$ which we denote by
$\mathrm{aco}(X)$.
\medskip

Suppose now that the processes $X$, $\hat X$ and $Y$ are defined as in (\ref{defX}), (\ref{defhatX}) and (\ref{defY}), respectively, where
for the construction of $\hat X$ and $Y$ we use the partition $(I_k)_{k\in \Z}$ given by (\ref{defIk}).
First we show that $\mathrm{aco}(X)$,  $\mathrm{aco}(\hat X)$, and  $\mathrm{aco}(Y)$ are the same sets up to a numeric constant.
Namely, the following is valid.
\begin{prop}
We have
\be \label{acoXXY}
   \mathrm{aco}(X)\subseteq \mathrm{aco}(\hat X)\subseteq 2\, \mathrm{aco}(Y) \subseteq  4\, \mathrm{aco}(\hat X) \subseteq 8\, \mathrm{aco}(X).
\ee
\end{prop}
\begin{proof}
By the definition of $\hat X$ it follows that
\be\label{more1}
   \mathrm{aco}(X)\subseteq \mathrm{aco}(\hat X)\subseteq 2 \, \mathrm{aco}(X).
\ee
On the other hand, (\ref{YtoX}) yields
\be  \label{more1a}
   \hat X  \subseteq \sum_{m=0}^\infty 2^{-m} Y,
\ee
hence
\be \label{more2}
   \mathrm{aco}(\hat X) \subseteq \sum_{m=0}^\infty 2^{-m} \mathrm{aco}(Y) = 2\, \mathrm{aco}(Y),
\ee
while (\ref{XtoY}) implies
$$
   Y \subseteq  \hat X -[0,1]\cdot \hat X,
$$
hence
\be\label{more3}
   \mathrm{aco}(Y) \subseteq 2\, \mathrm{aco}(\hat X).
\ee
By combining the inclusions (\ref{more1}), (\ref{more2}), (\ref{more3}), claim (\ref{acoXXY}) follows.
\end{proof}

\begin{remark}
\rm
Clearly, (\ref{acoXXY}) means that the three processes are either all bounded or all are unbounded. But,
certainly, it contains even more information.
\end{remark}
\medskip

Now we move to covering numbers in order to clarify the role of the distance $d$ defined in (\ref{defd}).
We will show now that on the {\it logarithmic} level there is no much difference between
 $N(X,||\cdot||_2, \eps)=N(T,d_{X},\eps)$ and $N(T,d,\eps)$.

 \begin{thm} \label{t51} We have
 $$
     \int_0^{\infty} \sqrt{\log N(T, d_{X},u)} \d u<\infty\quad
     \Leftrightarrow\quad
     \int_0^{\infty} \sqrt{\log N(T,d,u) }\d u<\infty\;.
 $$
 and
 $$
    \sup_{\eps>0}\,\eps^2\,\log N(T,d_{X}, \eps)  <\infty\quad
    \Leftrightarrow\quad
     \sup_{\eps>0}\,\eps^2\,\log N(T,d,\eps)<\infty\;.
 $$
\end{thm}

\begin{proof}
We first give the lower bounds for $N(T, d_{X},\eps)$.
By (\ref{defeta}) and (\ref{esteta0}) it follows that
\be  \label{more_nlb}
 N(X-[0,1]\cdot X,||\cdot||_2,\frac{\eps}{\sqrt{2}})
 \ge N(T,d,2\eps)-1.
\ee
Our next task is to replace $X-[0,1]\cdot X$ by $X$ in (\ref{more_nlb})
by using the following trivial fact.

\begin{lem}
Let $X$ be a subset of a normed space and $M_X:=\sup_{x\in X}||x||$. Then
\be \label{noaco1}
   N([0,1]\cdot X,\norm{\,\cdot\,},2\eps) \le N(X,\norm{\,\cdot\,},\eps)  \, \frac {M_X}{\eps}\quad\mbox{and}
\ee
\be \label{noaco2}
   N(X-[0,1]\cdot X,\norm{\,\cdot\,},3\eps) \le N(X,\norm{\,\cdot\,},\eps)^2  \, \frac {M_X}{\eps}\ .
\ee
\end{lem}

\begin{proof} [ of the lemma] Let $B$ be an $\eps$--net for $X$ and set
\[
 C=\left\{ \frac{j\eps}{M_X}\ y \;:\ y\in B,\;  j\in \N\,,\,  1\le j \le \frac{M_X}{\eps}\right\}.
 \]
 Clearly,
 \[
 \#\{C\}\le \#\{B\} \ \frac {M_X}{\eps}\ .
 \]
 Take any $z=\theta x\in [0,1]\cdot X$ with  $x\in X$, $\theta\in[0,1]$.
 Find  $y\in B$ and a positive integer $j \le \frac{M_X}{\eps}$ such that
 \[
      ||x-y||< \eps, \quad
      \left| j - \frac{\theta M_X}{\eps} \right| \le 1\;.
  \]
 Then $z':= \frac{j\eps}{M_X}\ y\in C$ and observe that
 \begin{eqnarray*}
    ||z-z'|| &=& \left\| \theta x- \frac{j\eps}{M_X} \ y\right\|
 \\
    &\le& \theta ||x-y|| + \left|\theta -\frac{j\eps}{M_X}\right|\ ||y||
 \\
 &<& \eps + \frac{\eps}{M_X} \, M_X =2\eps.
 \end{eqnarray*}
 Hence, $C$ is a $2\eps$--net for $[0,1]\cdot X$ and the first claim of the lemma is proved.
 The second one follows immediately.
\end{proof}

We may proceed now with the proof of Theorem \ref{t51}.
Combining (\ref{noaco2}) with (\ref{more_nlb}) leads to
$$
  N(T,d_X,\eps)^2\ge \frac {\eps}{M_X}
  \left[ N(T,d,6\sqrt{2}\eps)-1   \right].
$$
Hence, we conclude
 \[
     \int_0^{\infty} \sqrt{\log N(T, d_X,u)} \d u<\infty
   \quad  \Rightarrow\quad
     \int_0^{\infty} \sqrt{\log N(T,d,u) }\d u<\infty
 \]
 and
 \[
    \sup_{\eps>0}\eps^2\,\log N(T,d_X, \eps)  <\infty
   \quad \Rightarrow\quad \sup_{\eps>0}\eps^2\,\log N(T,d,\eps)<\infty\,.
 \]

Conversely, we will move now towards an upper bound for  $N(T,d_X,\eps)$.
By Proposition \ref{p4} and Proposition 3.2 of \cite{LifLin10} we have
\be \label{more_nub}
  N(T,d_Y,\eps)
  \le \tilde N(T,\hat d,\eps) +1
  \le \tilde N(T,d,\eps/2) +1
  \le N(T,d,\eps/4)+1.
\ee
Next, we trivially obtain from (\ref{more1a}) that
\begin{eqnarray*}
  N(T, d_{\hat X}, 2\eps)
  &\le&  N\Big(T, d_{\hat X}, \left(\sum_{m=0}^\infty (m+1)^{-2}\right) \eps \Big)
  \\
   &\le& \prod_{m=0}^\infty N(T, d_{2^{-m} Y}, (m+1)^{-2} \eps)
  \\
   &=& \prod_{m=0}^\infty N(T, d_Y, 2^m (m+1)^{-2} \eps).
\\
   &\le& \prod_{m=0}^\infty N_*( T, d, 2^{m-2} (m+1)^{-2} \eps),
\end{eqnarray*}
where we used (\ref{more_nub}) on the last step and
\[
   N_*( T, d, r) :=
   \left\{
   \begin{array}{ccc}
   N( T, d, r)+1 &:&  N( T, d_Y, 4r)>1,
   \\
   1 &:&  N( T, d_Y, 4r)=1.
   \end{array}
   \right.
\]
It follows that
\be \label{more5}
  \log N(T, d_{\hat X}, 2\eps)
  \le \sum_{\{m\ge 0:\, 2^m (m+1)^{-2} \eps \le M_Y\}} \log \left( N(T,d,2^{m-2} (m+1)^{-2} \eps)+1\right) ,
\ee
where $M_Y:=\sup_{t\in T} ||Y_t||_2$.
For the Dudley integral this implies
\begin{eqnarray*}
  \int_0^\infty \sqrt{\log N(T, d_{\hat X}, 2\eps)}  \d\eps
  &\le& \sum_{m=0}^\infty  \, \int_0^{\frac{M_Y}{2^m (m+1)^{-2}} }
  \sqrt{\log \left(N(T,d,2^{m-2} (m+1)^{-2} \eps)+1\right)} \d\eps
  \\
  &\le&   \sum_{m=0}^\infty  \frac{(m+1)^2}{2^{m-2}}  \,
  \int_0^{\infty} \sqrt{\log\left(N(T,d,u)+1\right)} \d u
  \\
  &=& C\ \int_0^{\infty} \sqrt{\log \left(N(T,d,u)+1\right)} \d u\,.
\end{eqnarray*}
Hence,
 \[
    \int_0^{\infty} \sqrt{\log N(T,d,u) }\d u<\infty
    \quad\Rightarrow\quad
    \int_0^{\infty} \sqrt{\log N(T, d_{\hat X},u)} \d u<\infty\, .
 \]
 Moreover,  (\ref{more5}) yields
 \[   \sup_{\eps>0}\eps^2\,\log N(T,d,\eps)<\infty
      \quad\Rightarrow\quad
      \sup_{\eps>0}\eps^2\,\log N(T, d_{\hat X}, \eps)  <\infty\,.
 \]
The final passage goes from $\hat X$ to $X$. Since $X\subseteq [0,1]\cdot \hat X$,
by applying (\ref{noaco1}) to $\hat X$ we obtain
\[
  \log N(T, d_X,2\eps) \le \log N([0,1]\cdot \hat X,||\cdot||_2,2\eps)
  \le \log N(T, d_{\hat X},\eps) + \log\left(\frac{M_{\hat X}} {\eps}\right).
\]
Hence
\[
     \int_0^{\infty} \sqrt{\log N(T,d_{\hat X},u)} \d u<\infty
   \quad  \Rightarrow\quad
     \int_0^{\infty} \sqrt{\log N(T,d_X,u)} \d u<\infty
 \]
 as well as
 \[   \sup_{\eps>0}\eps^2\,\log N(T,d_{\hat X},\eps)<\infty
      \quad\Rightarrow\quad
      \sup_{\eps>0}\eps^2\,\log N(T,d_X,\eps)<\infty.
 \]
By combining the preceding estimates we finish the proof.
\end{proof}

\section{The binary tree with homogeneous weights}
\label{Binary}
\setcounter{equation}{0}

Before investigating Gaussian processes on binary trees let us shortly recall some basic facts
about suprema of Gaussian sequences.

Let $(X_1,\dots, X_n)$ be a centered Gaussian random vector. Introduce the following notations:
$$
\s_1^2:=\min_j \E X_j^2,\quad \s_2^2:=\max_j \E X_j^2,\quad S:=\max_j X_j\;,
$$
and
let $m_S$ be a median of $S$.
Then the following is well known.
\begin{itemize}
\item
It is true that
\be \label{mSES}
   m_S\le \E S.
\ee
See \cite{Lif}, p.143.

\item The following concentration principle is valid:
\[
  \P(S>m_S+r) \le \hat\Phi(r/\s_2)\le  \exp(-r^2/2\s_2^2), \qquad \forall r>0,
\]
where
\[
 \hat\Phi(r)=\frac{1}{\sqrt{2\pi}}\int_r^{\infty} e^{-\frac{u^2}{2}}\d u
\]
is the standard Gaussian tail.  See \cite{Lif}, p.142.
By combining this with (\ref{mSES}) we also have
\be \label{tailS}
 \P(S>\E S+r) \le   \exp(-r^2/2\s_2^2), \qquad \forall r>0.
\ee

\item It is true that
\be \label{ESub}
 \E S\le  \sqrt{2\log n}\ \s_2\, .
\ee
See \cite{Lif}, p.180.

\item
If $X_1,\ldots,X_n$ are independent, then
\be \label{sudakov}
 \E S\ge c \sqrt{\log n}\ \s_1 \, .
\ee
with $c=0.64$. See \cite{Lif}, p.193--194.
\end{itemize}
Remark that the same properties hold true for
 \[
   S':= \max_{j\le n} |X_j| = \max_{j\le n}\max\{X_j,-X_j\}.
 \]
\bigskip

Let $T$ be a binary tree and suppose that the weights depend only on the level
numbers, i.e.~$\al(t)=\al_{|t|}$ and $\s(t)=\s_{|t|}$ for some sequences $(\al_k)_{k\ge 0}$
and $(\s_k)_{k\ge 0}$ of positive numbers with $(\s_k)_{k\ge 0}$ non--increasing.
The following two theorems give, with a certain overlap, necessary and sufficient
conditions for the boundedness of  $(X_t)_{t\in T}$ in that case.

\begin{thm} \label{t61}
a) If $X=(X_t)_{t\in T}$ is a.s.~bounded, then
\be \label{sumal}
 G:=\sup_n \ \s_n \, \sum_{k=1}^n \al_k <\infty.
\ee
b) Moreover, if $(\al_k)_{k\ge 0}$ satisfies the regularity assumption
\be \label{regal}
 Q:= \sup_n \sup_{n\le k\le 2n} \frac{\al_k}{\al_n} <\infty,
\ee
then $X$ is a.s.~bounded if and only if $(\ref{sumal})$ holds.
\end{thm}

\begin{proof}
a)
 Let us construct a random sequence $(t_n)_{n\ge 0}$ in $T$ and a sequence of random variables $(\zeta_n)_{n\ge 1}$
by the following inductive procedure. Let $t_0=\ro$. Next, assuming that $t_n$ is constructed,
let $t'$ and $t''$ be the two offsprings of $t_n$.
We let
\[
   \zeta_{n+1}:= \max\{\xi_{t'}, \xi_{t''}\}, \qquad t_{n+1}:= \textrm{argmax} \{\xi_{t'}, \xi_{t''}\}.
\]
It is obvious that $(\zeta_n)$ are i.i.d. random variables with strictly positive expectation.
Our construction yields
 \[
    X_{t_n}= \s_n \left( \al_0 \xi_{\ro}+ \sum_{j=1}^n  \al_j\ \zeta_{j} \right) , \qquad n\ge 1 .
 \]
 It follows that
 \[
  \E \sup_{t\in T} X_t \ge  \sup_{n\ge 1} \E X_{t_n} =C\ \sup_{n\ge 1}\,  \s_n  \, \sum_{j=1}^n  \al_j,
 \]
 where $C:=\E \zeta_j>0$.
Since the assumption "$(X_t)_{t\in T}$ is a.s.~bounded" implies  $\E \sup_{t\in T} X_t<\infty$,
 we obtain (\ref{sumal}).
\medskip

b) Let us assume that $G<\infty$, $Q<\infty$ and prove that $(X_t)_{t\in T}$ is a.s.~bounded.
For any $m\ge 0$ set  $B_m=[2^m,2^{m+1})$ and $J_m:=\{t\in T: |t|\in B_m\}$.
For any $M\ge 0$ and $t\in J_M$ write
\be \label{Ubound1}
  \sum_{v\preceq t} \al(v)\xi_v = \sum_{m=0}^M  \sum_{{v\preceq t \atop v\in J_m }} \al(v)\xi_v
  \le   \sum_{m=0}^M U_m,
\ee
where
\[
  U_m := \sup_{u\in J_m} \left|\sum_{{v\preceq u \atop v\in J_m }} \al(v)\xi_v \right|.
\]
By using that $(\s_k)_{k\ge 0}$ is non--increasing, we infer from (\ref{Ubound1})
for any $M\ge 0$ and $t\in J_M$
\begin{eqnarray*}
 X_t &=&\sigma_t \sum_{v\preceq t} \al(v)\xi_v   \le \sigma_{2^M}  \sum_{m=0}^M U_m
\\
&=& \sigma_{2^M}  \sum_{m=0}^M (\E U_m+ (U_m-\E U_m))
\\
&\le& \sigma_{2^M}  \sum_{m=0}^M (\E U_m+ (U_m-\E U_m)_+)
\\
&\le& \sigma_{2^M}  \sum_{m=0}^M \E U_m
+    \sum_{m=0}^\infty \sigma_{2^m} (U_m-\E U_m)_+.
\end{eqnarray*}
Hence,
\be \label{twoparts}
\sup_{t\in T} X_t \le \sup_{M\ge 0}  \sigma_{2^M}  \sum_{m=0}^M \E U_m
  + \sum_{m=0}^\infty \sigma_{2^m} (U_m-\E U_m)_+ .
\ee
We will use now standard Gaussian techniques in order to evaluate the
quantities on the r.h.s.
Note that on the binary tree
\[
   \#\{J_m\}\le \#\{t:|t|<2^{m+1}\} \le 2^{2^{m+1}}.
\]
Moreover, we have
\[
h_m^2:=   \sup_{u\in J_m}   \sum_{{v\preceq u \atop v\in J_m }} \al(v)^2 \le \sum_{k\in B_m} \al_k^2.
\]
Assuming  (\ref{regal}) to hold, we obtain
\[
  h_m^2\le \sum_{k\in B_m} \al_k^2 \le Q^2 \ 2^m\, \al_{2^m}^2.
\]
Using (\ref{regal}) again we arrive at
\be \label{hmbound}
  h_m\le Q \ 2^{m/2}\, \al_{2^m}\le
  Q^2 \ 2^{1-m/2}\,  \sum_{k\in B_{m-1}} \al_k.
\ee
Now by (\ref{ESub}) it follows that
\[
\E U_m \le \sqrt{\log(2\#\{J_m\})}\, h_m
\le 4 Q^2 \sum_{k\in B_{m-1}} \al_k.
\]
Hence, for any $M$ we get
\[
    \sigma_{2^M}  \sum_{m=0}^M \E U_m
    \le  \sigma_{2^M} 4Q^2 \sum_{m=0}^M \sum_{k\in B_{m-1}} \al_k
    = \sigma_{2^M} 4Q^2 \sum_{k=0}^{2^M-1} \al_k
    \le  4 Q^2 G.
\]
On the other hand, by the Gaussian concentration principle  (\ref{tailS}),
\[
\E(U_m-\E U_m)_+  = \int_0^\infty \P(U_m-\E U_m >r) \d r \le
\int_0^\infty \exp(-r^2/2h_m^2) \d r \le 2 h_m .
\]
From (\ref{hmbound}) it follows that
\begin{eqnarray*}
\sigma_{2^m} \E(U_m-\E U_m)_+
&\le& 2 \sigma_{2^m} h_m
\\
&\le& 2 \sigma_{2^m}  Q^2 \  2^{1-m/2}\,  \sum_{k\in B_{m-1}} \al_k.
\\
&\le&  2^{2-m/2} Q^2 \sigma_{2^m}  \sum_{k\le 2^{m}} \al_k
\\
&\le&  2^{2-m/2} Q^2 G.
\end{eqnarray*}
By plugging this into (\ref{twoparts}), we arrive at
\[
\E \sup_{t\in T} X_t \le \sup_{M\ge 0}  \sigma_{2^M}  \sum_{m=0}^M \E U_m
  + \sum_{m=0}^\infty \E \sigma_{2^m} (U_m-\E U_m)_+
\le  4 Q^2 G + Q^2 G \sum_{m=0}^\infty 2^{2-m/2}
<\infty
\]
and $(X_t)_{t\in T}$ is a.s.~bounded.
\end{proof}
\medskip

Let us start with a first example where Theorem \ref{t61} applies.
Take the binary tree $T$ and suppose that either $\al(t)=(|t|+1)^{-1}$ and $\s(t)\equiv 1$ or that
$\al(t)\equiv 1$ and $\s(t)=(|t|+1)^{-1}$. Note these weights lead to critical cases, namely,
we have  $\log N(T,d,\eps)\approx \eps^{-2}$ for both pairs of weights.

\begin{cor}
\label{bin}
The process
$$
X'_t:=(|t|+1)^{-1}\sum_{v\pq t}\xi_v\,,\quad t\in T\,,
$$
is a.s.~bounded while
\[
   X''_t:=\sum_{v\pq t}(|v|+1)^{-1}\xi_v\,,\quad t\in T\,,
\]
is a.s.~unbounded.
\end{cor}

\begin{proof}
In the first case (\ref{sumal}) and (\ref{regal}) are satisfied while in the second one
(\ref{sumal}) fails. Thus both assertions follow by Theorem \ref{t61}.
\end{proof}

\begin{remark}
\rm
The preceding corollary is of special interest because $\al(t)\s(t)=(|t|+1)^{-1}$
in both cases. Consequently, the boundedness of the process $X$ cannot be described by the
behavior of $\al\s$.
This is in contrast to the main results about metric entropy in \cite{LifLin10} which
only depend on this product behavior.
\end{remark}
\medskip

Theorem \ref{t61} does not apply in the case of rapidly increasing sequences $(\al_k)_{k\ge 0}$ because
(\ref{regal}) fails for them. The next theorem fills this gap.
\begin{thm} \label{t62}
a) If $X=(X_t)_{t\in T}$ is a.s.~bounded, then
\be \label{sumal2}
    G_1 :=\sup_n \sup_{m\le n} \ \s_n \,  \sqrt{m} \left(\sum_{k=m}^n \al_k^2\right)^{1/2} <\infty.
\ee

b) If
\be \label{sumal2a}
     G_2 :=\sup_n  \ \s_n \,  \sqrt{n} \left(\sum_{k=0}^n \al_k^2\right)^{1/2} <\infty,
\ee
then
$(X_t)_{t\in T}$ is a.s.~bounded.

c) Moreover, if $(\al_k)_{k\ge 0}$ is non--decreasing,
then the conditions $(\ref{sumal2})$ and $(\ref{sumal2a})$ are equivalent,
thus $X$ is a.s.~bounded if and only if either of them  holds.
\end{thm}

\begin{proof}
a) Let us fix a pair of integers $m\le n$. Take any mapping $L:\{t:|t|=m\}\to \{t:|t|=n\}$
such that $t \preceq  L(t)$ for all $t$.
Consider
\[
  Y_t:=\s_n \sum_{t \preceq s \preceq L(t)} \al_{|s|} \xi_s, \qquad |t|=m.
\]
Notice that the $(Y_t)_{|t|=m}$ are independent and that
\[
   \E Y_t^2 =\s_n^2 \sum_{m \le k\le n} \al_k^2 \, .
\]
By (\ref{sudakov}) it follows
\[
  \E \max_{|t|=m} Y_t \ge c \sqrt{\log(2^m)}\ \s_n \left( \sum_{m\le k\le n} \al_k^2\right)^{1/2}
  =  \tilde c \ \sqrt{m}\ \s_n \left( \sum_{m\le k\le n} \al_k^2\right)^{1/2}.
\]
On the other hand
\[
   Y_t= X_{L(t)}-\frac{\s_n}{\s_m} \, X_t \, ,
\]
hence
\[
    \max_{|t|=m} Y_t \le 2 \sup_{t\in T} |X_t|.
\]
We arrive at
\[
   2 \ \E \sup_{t\in T} |X_t|
   \ge \tilde c \ \sqrt{m}\ \s_n \left( \sum_{m\le k\le n} \al_k^2\right)^{1/2},
\]
and achieve the proof of a) by taking the supremum over $m$ and $n$.
\medskip

b) Let $S_n:=\max_{|t|=n} X_t$. By (\ref{ESub}) we have
\be \label{ESnub}
  \E S_n \le \sqrt{2\log(2^n)} \ \s_n  \left( \sum_{k=0}^n \al_k^2\right)^{1/2} \le 2 G_2.
\ee
We also have
\be \label{EX2ub}
 \E X_t^2= \s_n^2 \sum_{k=0}^n \al_k^2 \le \frac{G_2^2}{n}, \qquad |t|=n.
\ee
Since $\sup_{t\in T} X_t =\sup_n S_n$, for any $r>0$ it follows that
\begin{eqnarray*}
  \P\left( \sup_{t\in T} X_t> 2G_2+r\right)
  &\le& \sum_{n=0}^\infty \P\left( S_n \ge 2G_2+r\right)
\\
  &\le& \sum_{n=0}^\infty \P\left( S_n \ge \E S_n +r\right)
  \qquad (\textrm{by}\ (\ref{ESnub}) \, )
\\
   &\le& \P\left( S_0 \ge \E S_0 +r\right)  +
         \sum_{n=1}^\infty \exp \left(-\frac{r^2 n}{2 G_2^2} \right)
   \qquad (\textrm{by}\ (\ref{EX2ub})\ \textrm{and}\ (\ref{tailS}) \, )
\\
   &=&  \P\left( X_\ro \ge r \right) +
   \frac {\exp \left(-\frac{r^2}{2G_2^2}\right)}{1- \exp \left(-\frac{r^2}{2G_2^2}\right)} \to 0,  \qquad \textrm{as}\ r\to \infty.
\end{eqnarray*}
It follows that $(X_t)_{t\in T}$ is a.s.~bounded. Thus assertion b) is proved.
\medskip

c) The inequality $G_1\le G_2$ is obvious for any $(\al_k)_{k\ge 0}$.
We only need to show that a bound in the opposite direction holds,
too. Let
\[
  m_n:=\left\{\begin{array}{ccc}
  \frac n2 &:& n\ \textrm{even}\\
  \frac {n+1}2 &:& n\ \textrm{odd.}
  \end{array}
  \right.
\]
Assuming that $(\al_k)_{k\ge 0}$ is non-decreasing, we have
\[
  \sum_{m_n \le k \le n} \al_k^2 \ge  \sum_{0\le k <m_n} \al_k^2 ,
\]
hence
\[
 2 \sum_{m_n \le k \le n} \al_k^2 \ge  \sum_{0\le k \le n} \al_k^2 .
\]
It follows that
\[
 G_1 \ge \sup_n  \ \s_n \,  \sqrt{m_n} \left(\sum_{m_n \le k\le n} \al_k^2\right)^{1/2}
 \ge \frac 12\  \sup_n  \ \s_n \,  \sqrt{n}  \sum_{0\le k \le n} \al_k^2
 =\frac {G_2}{2}.
\]
\end{proof}

\begin{cor}
\label{c2}
Let $\al_k=k^b\, 2^k$ for some $b\in\R$. Then $(X_t)_{t\in T}$ is a.s.~bounded if
and only if
\[
   \sup_n  \ \s_n \, n^{1/2+b}\,  2^n <\infty.
\]
\end{cor}
\begin{remark}
\rm
Note that criterion (\ref{sumal}) from Theorem \ref{t61} fails to work in that case. Moreover, letting
$b=-\gamma$ with $1/2<\gamma< 1$ and $\s_n=2^{-n}$, by Corollary \ref{c2} the corresponding process is bounded
although $\al(t)\s(t)\ge |t|^{-\gamma}$ for $t\in T$. This shows that
the second part of Corollary \ref{c1} is no longer valid for non--constant weights $\s$.
\end{remark}
\medskip

Another example where Theorem \ref{t61} does not apply is as follows.

\begin{cor} \label{cor65}
Let $\al_k^2=\exp((\log k)^\beta)$ with $\beta>1$.
Then $(X_t)_{t\in T}$ is a.s.~bounded if and only if
\[
   \sup_n  \ \s_n \  \frac{n}{(\log n)^{\frac{\beta-1}{2}}} \  \exp((\log n)^\beta/2)  <\infty.
\]
\end{cor}

\begin{proof} Easy calculation shows that
\begin{eqnarray*}
\sum_{k=0}^n \al_k^2 &\sim& \int_1^n \exp((\log u)^\beta) du
= \int_0^{(\log n)^\beta} \exp(z+z^{1/\beta})\frac{dz}{\beta z^{1-1/\beta}}
\\
&\sim&
 \frac{n}{\beta (\log n)^{\beta-1}} \  \exp((\log n)^\beta).
\end{eqnarray*}
An application of Theorem \ref{t62} yields the result.
\end{proof}
\medskip

Our message is that Theorems \ref{t61} and \ref{t62} should {\it jointly} cover any reasonable case.
Let us illustrate this by the following example.
Recall that by the first part of Corollary \ref{c1}, if $T$
is the binary tree and
$\al(t)\s(t)\le c\,|t|^{-\gamma}$ for some $\gamma>1$, then the generated process $X$ is a.s.~bounded.
For homogeneous (level--dependent) weights this means that $\al_k\s_k\le c\,k^{-\gamma}$ for some
$\gamma>1$ yields the a.s.~boundedness of $X$. Let us see how this fact is related to Theorems \ref{t61}
and \ref{t62}.

Essentially, we have the following
\begin{itemize}
\item If $(\al_k)_{k\ge 0}$ is decreasing, then
\[
  \sigma_n \sum_{k\le n} \al_k \le  \sigma_n \sum_{k\le n} \frac{c\ k^{-\gamma}}{\s_k} \le
  \sum_{k\le n} c \ k^{-\gamma} \le   c \ \sum_{k=1}^\infty  k^{-\gamma},
\]
hence (\ref{sumal}) and  (\ref{regal}) hold and Theorem \ref{t61} yields the boundedness.

\item If $(\al_k)_{k\ge 0}$ is increasing, then
\[
  \sigma_n  \sqrt{n} \left(\sum_{k\le n} \al_k^2\right)^{1/2}
  \le \sigma_n  \sqrt{n} \left( n \al_n^2\right)^{1/2}
  = \sigma_n \al_n \ n \le c\ n^{1-\gamma},
\]
thus (\ref{sumal2a}) holds even for $\gamma\ge 1$,
and Theorem \ref{t62} yields the boundedness.
\end{itemize}
\medskip

Finally let us relate the results in Theorems \ref{t61} and \ref{t62} to those about
compactness properties of $(T,d)$ with $d$ defined in (\ref{defd}). Here we have the following partial
result.

\begin{prop}
\label{prop3}
The expression $G_1$ in $(\ref{sumal2})$  is finite if and only if there is
a constant $c>0$ such that
\be
\label{d1}
d(t,s)\le c\,|t|^{-1/2}
\ee
for all $t,s\in T$ with $t\pc s$.
\end{prop}
\begin{proof}
First note that in the case of homogeneous weights we get
$$
d(t,s)=\max_{|t|<l\le |s|} \s_l\left(\sum_{k=|t|+1}^l\al_k^2\right)^{1/2}\;.
$$
Next we remark that $G_1<\infty$ if and only if there is a constant $c>0$ such that
\be
\label{G1}
\s_n\left(\sum_{k=m+1}^n \al_k^2\right)^{1/2}\le c\,m^{-1/2}
\ee
for all $0\le m<n<\infty$.

Suppose now that (\ref{d1}) holds and take integers $m<n$. Next choose two elements $t,s\in T$ with $t\pc s$  such that
$m=|t|$ and $n=|s|$. Note that (\ref{d1}) implies
$$
\s_n\left(\sum_{k=m+1}^n \al_k^2\right)^{1/2}\le d(t,s)\le c\,|t|^{-1/2}= c\, m^{-1/2}
$$
which proves (\ref{G1}).

Conversely, assume (\ref{G1}) and take any two elements $t\pc s$ in $T$. Furthermore, let $v\in(t,s]$ be a node
where
$$
d(t,s)= \s_{|v|}\left(\sum_{k=|t|+1}^{|v|}\al_k^2\right)^{1/2}\;.
$$
Applying (\ref{G1}) with $m:=|t|$ and $n:=|v|$ leads to
$$
d(t,s) \le c\, m^{-1} =c\,|t|^{-1/2}
$$
as claimed. This completes the proof.
\end{proof}
\medskip

\begin{remark}
\rm
Clearly (\ref{d1}) implies $\log N(T,d,\eps)\le c\,\eps^{-2}$ as we already know by combining Theorems
\ref{DS} and \ref{t62}. But it says a little bit more. Namely, an $\eps$--net giving this
order may be chosen as $\set{t\in T : |t|\le c\,\eps^{-1/2}}$ for a certain $c>0$. Of course, this
heavily depends on the fact that we deal with homogeneous weights.
\end{remark}
\bigskip

{\bf Acknowledgement.} The research was supported by the RFBR-DFG grant 09-01-91331 "Geometry
and asymptotics of random structures". The work of the first named author was also supported by RFBR grant
10-01-00154a, as well as by Federal Focused Programme 2010-1.1-111-128-033.

\bibliographystyle{amsplain}

\vspace{1cm}

\parbox[t]{7cm}
{Mikhail Lifshits\\
St.Petersburg State University\\
Dept Math. Mech. \\
198504 Stary Peterhof, \\
Bibliotechnaya pl., 2\\
Russia\\
email: lifts@mail.rcom.ru}\hfill
\parbox[t]{6cm}
{Werner Linde\\
Friedrich--Schiller--Universit\"at Jena \\
Institut f\"ur Stochastik\\
Ernst--Abbe--Platz 2\\
07743 Jena\\
Germany\\
email: werner.linde@uni-jena.de}
\end{document}